\documentclass{amsart}
\usepackage{amssymb,latexsym,amsmath,amsfonts,cite}

\advance\textheight by \topskip
\newtheorem{theorem}{Theorem}[section]
\newtheorem{corollary}[theorem]{Corollary}
\newtheorem{lem}[theorem]{Lemma}
\numberwithin{equation}{section}

\newcommand{\Z}{\mathbb{Z}}
\newcommand{\F}{\text{Flatten}}
\newcommand{\rmand}{\quad\hbox{ and }\quad}
\def\={\;=\;}
\def\l{\;\le\;}

\allowdisplaybreaks

\title[Counting 13-2 in flattened permutations]{Recurrence relations in counting the pattern 13-2 in flattened permutations}

\author{Toufik Mansour}
\address{Department of Mathematics, University of Haifa, 3498838 Haifa, Israel}
\email{tmansour@univ.haifa.ac.il}

\author{$^\dag$David G.L. Wang}

\thanks{$^\dag$Corresponding author, e-mail: david.combin@gmail.com}

\address{School of Mathematics and Statistics, Beijing Institute of Technology, 102488 Beijing, P.R. China}
\email{david.combin@gmail.com}



\begin{document}

\begin{abstract}
We prove that the generating function for the number of flattened permutations
having a given number of occurrences of the pattern 13-2 is rational,
by using the recurrence relations and the kernel method.
\end{abstract}
\maketitle
\tableofcontents

\section{Introduction}
Problems on {\em pattern avoidance} have received much attention during the past decades.
Extensive studies on the pattern avoidance of permutations
can be found from both enumerative combinatorics and algebraic combinatorics; see \cite{HM,Ki,M}.

A permutation of length $n$ is a rearrangement of the letters $1,2,\ldots,n$.
Classically, a {\em pattern} is a permutation.
An {\em occurrence} of a pattern in a permutation is a subsequence of the permutation,
which is order-isomorphic to the pattern.
While most of the previous studies on permutation patterns have focused on pattern avoidance,
relatively few work has been on the enumeration of a pattern of length~$2$ or length~$3$; see~\cite{Nak13D}.
For example,
the~$6$ patterns of length~$3$ are classified into two classes subject to the pattern counting;
see \cite{NZ96} for the representative pattern~$123$,
and \cite{MV02,NZ96} for the representative pattern~$132$.

Recently,
Callan \cite{C} introduced the notion of flattened set partitions.
Translating the idea of flattening a combinatorial structure onto permutations,
Mansour and Shattuck~\cite{MS} obtained equally beautiful enumerative results.
The {\em standard cycle form} of a permutation is a product representation by cycles of letters,
where the smallest letter of each cycle is placed at the first position of that cycle,
and where the cycles are placed in the increasing order with respect to the smallest letters.
The {\em flattened permutation} associated with the permutation~$\pi$,
denoted by $\F(\pi)$, is defined to be the permutation obtained by erasing the parentheses
enclosing the cycles of the permutation~$\pi$, which is in its standard cycle form.
For example, the permutation $\pi=71564328$ has the standard cycle form $(172)(3546)(8)$.
Thus its associated flattened permutation is $\F(\pi)=17235468$.
Taking the number of peaks and valleys into account,
Mansour, Shattuck and Wang~\cite{MSW} obtained the generating functions of the flattened permutations.

A permutation $\pi_1\pi_2\cdots \pi_n$ of length~$n$ is said to contain the pattern 13-2 (for other patterns of length~$3$, see \cite{MSW14}), if there are two indices $i,j$ such that $2\leq i<j\leq n$ and that $\pi_{i-1}<\pi_j<\pi_i$. In this paper, counting the pattern 13-2 in flattened permutations, we obtain the recurrence
\[
g_n
=\sum_{j=1}^{n-1}\sum_{k=0}^{n-1-j}\frac{n-1+j-k}{j}\binom{j+k-2}{j-2}\binom{n-2-k}{j-1}q^k(q-1)^{j-1}g_{n-j},
\]
for the generating function
\[
g_n=\sum_{\pi\in\mathcal S_n}q^{\text{the number of occurrences of the pattern 13-2 in Flatten}(\pi)}.
\]
In particular, it implies that the average number of occurrences of the pattern 13-2 can be expressed in terms of the harmonic numbers. As another consequence, the number of permutations of length $n$ whose associated flattened permutation avoids the pattern 13-2 is $2^{n-1}$. By using the recurrence, we also show the rationality of the generating function of flattened permutations with a given number of occurrences of the pattern 13-2.

\bigskip

\section{Counting 13-2 patterns in flattened permutations}
Let $n\ge1$ and $\pi\in\mathcal{S}_n$. For any $1\le k\le n$, define the generating function
\[
g_n(a_1a_2\cdots a_k)
=\sum_{\pi}q^{\text{the number of occurrences of the pattern 13-2 in Flatten}(\pi)},
\]
where $\pi$ ranges over all permutations of length~$n$ such that
the flattened permutation~$\F(\pi)$ starts with the subword $a_1a_2\cdots a_k$.
For example, we have $g_3(12)=4$ and $g_3(13)=2q$.
By the definition of flattened permutations, we have $g_n(a_1a_2\cdots a_k)=0$ if $a_1\ne1$.
Let
\[
g_n=g_n(1).
\]
Then $g_n$ is the generating function of the number of occurrences of the pattern 13-2 in flattened permutations of length~$n$.
For example, $g_1=1$, $g_2=2$, and $g_3=4+2q$.
The following lemma will be used frequently in the sequel.

\begin{lem}\label{lem:g12:13-2}
We have $g_n(12)=2g_{n-1}$ for all $n\ge2$.
\end{lem}

\begin{proof}
Let $n\ge 2$ and $r\ge 0$.
Let $A_n=A_n(r)$ be the set of permutations $\pi$ of length $n$
such that the permutation $\F(\pi)$ has exactly $r$ occurrences of the pattern 13-2.
Let $B_n=B_n(r)$ be the subset of $A_n$ such that each permutation $\pi\in B_n$ satisfies that $\F(\pi)$ starts from the letters~$12$.
It suffices to construct a one-to-two correspondence between the set $A_{n-1}$ and the set $B_n$,
since then we will have $2|A_{n-1}(r)|=|B_n(r)|$.
In fact, if we obtain the correspondence for all $r\ge0$, then we have
\[
g_n(12)
\=\sum_{r\ge 0}|B_n(r)|q^r
\=\sum_{r\ge 0}2|A_{n-1}(r)|q^r
\=2g_{n-1}.
\]

Let $\sigma\in A_{n-1}$ be a permutation represented in the standard cycle form.
We add the number~$1$ to each letter in the permutation $\sigma$,
and denote the resulting arrangement of the set $\{2,3,\ldots,n\}$ by $\sigma'$.

Inserting the single cycle $(1)$ before all the cycles of $\sigma'$, we obtain a permutation of length $n$, denoted~$\pi$.
Denote $\F(\sigma)=\sigma_1\sigma_2\cdots\sigma_{n-1}$ and $\F(\pi)=\pi_1\pi_2\cdots\pi_n$.
Then we have $\sigma_1=\pi_1=1$ and $\pi_{k+1}=\sigma_k+1$ for all $k\in[n-1]$.
Therefore, the subsequence $\sigma_i\sigma_{i+1}\sigma_j$ ($1\le i\le n-2$ and $i+2\le j\le n-1$) is an occurrence of the pattern 13-2
if and only if the subsequence $\pi_{i+1}\pi_{i+2}\pi_{j+1}$ is an occurrence of the same pattern.
Since the permutation $\F(\pi)$ starts from the letters~$12$, it has no occurrences of the pattern 13-2 starting from the letter~$1$.
Thus, the above bijection implies that the number of occurrences of the pattern 13-2 in the permutation~$\F(\pi)$ equals
the number of occurrences of the same pattern in the permutation~$\F(\sigma)$. In other words, we have $\pi\in B_n$.

On the other hand, placing the letter~$1$ at the first position inside the first cycle of the arrangement~$\sigma'$,
we obtain another permutation of length~$n$, denoted by~$\pi'$.
Suppose that $\F(\pi')=\pi_1'\pi_2'\cdots\pi_n'$. Then we have $\pi'_1=1$ and $\pi'_{k+1}=\sigma_k+1$ for all $k\in[n-1]$.
It follows that $\F(\pi')=\F(\pi)$, and consequently, $\pi'\in B_n$.

Conversely, let $\tau\in B_n$ be a permutation represented in the standard cycle form.
Subtracting the number~$1$ from each letter of the permutation~$\tau$ gives an arrangement of the letters $0,1,\ldots,n-1$.
Removing the letter~$0$ from this arrangement, we obtain a permutation of length $n-1$, denoted by~$\mu$.
In similar fashion we can deduce that $\mu\in A_{n-1}$.
Hence, the map $\sigma\mapsto\{\pi,\pi'\}$ is a desired one-to-two correspondence. This completes the proof.
\end{proof}

For convenience,
we introduce the characteristic function~$\chi$ which is defined by
\[
\chi(P)=\begin{cases}
1,&\text{if $P$ is true},\\
0,&\text{if $P$ is false},
\end{cases}
\]
for any proposition $P$.

\begin{theorem}\label{thm:gn:13-2}
The generating function $g_n$ satisfies the recurrence relation
\begin{equation}\label{rec:gn:13-2}
g_n
=\sum_{j=1}^{n-1}b_{n,j}(q-1)^{j-1}g_{n-j}\qquad\text{for all $n\ge2$},
\end{equation}
where $b_{n,j}$ is the following polynomial in $q$:
\begin{equation}\label{def:b:13-2}
b_{n,j}
=\sum_{k=0}^{n-1-j}\frac{n-1+j-k}{j}\binom{j+k-2}{j-2}\binom{n-2-k}{j-1}q^k.
\end{equation}
\end{theorem}

\begin{proof}
When $n=2$, it is routine to check the truth since $g_1=1$ and $g_2=2$. Suppose that $n\ge3$.

Let $3\le i\le n$.
Let $\pi$ be a permutation of length~$n$ such that the permutation $\text{Flatten}(\pi)$ starts
with the letters~$1ij$.
Let $\pi'$ be the permutation obtained by removing the letter~$i$ from the permutation~$\pi$,
and subtracting each letter larger than~$i$ by the number~$1$.
From this definition,
we see that an occurrence of any pattern in the permutation~$\F(\pi)$ starting with the $k$th letter ($k\ge3$)
is an occurrence of the same pattern in the permutation~$\F(\pi')$ starting with the $(k-1)$th letter.
Consequently, the difference between the numbers of occurrences of the pattern 13-2 in the permutations~$\F(\pi)$ and~$\F(\pi')$
is equal to the difference between the number of occurrences of the pattern 13-2 in $\F(\pi)$ starting with the first $2$ letters
and the number of such occurrences in $\F(\pi')$ starting with the first letter.

For any permutation~$\sigma$,
let $f_k(\sigma)$ be the number of occurrences of the pattern 13-2 in the permutation~$\F(\sigma)$,
where each occurrence starts from its $k$th letter. We will compute the~$3$ numbers
\[
f_1(\pi),\qquad f_2(\pi),\rmand f_1(\pi').
\]
Since the occurrences of the pattern 13-2 in the permutation~$\F(\pi)$ starting with the first letter
are the $i-2$ subsequences $1i2$, $1i3$, $\ldots$, and $1i(i-1)$, we have $f_1(\pi)=i-2$.

If $j<i$, then the second letter~$i$ of the permutation~$\F(\pi)$ does not starts any occurrence of the pattern 13-2, namely, $f_2(\pi)=0$.
In this case, the occurrences of the pattern 13-2 in the permutation~$\F(\pi')$ starting with the first letter
are the $j-2$ subsequences $1j2$, $1j3$, $\ldots$, and $1j(j-1)$. Thus $f_1(\pi')=j-2$.
Therefore, there are exactly
\[
f_1(\pi)+f_2(\pi)-f_1(\pi')
\=(i-2)+0-(j-2)
\=i-j
\]
more occurrences of the pattern 13-2 in the permutation $\F(\pi)$ than in the permutation\break $\F(\pi')$.
In terms of generating functions, we have
\begin{equation}\label{eq:j<i:13-2}
g_n(1ij)\=q^{i-j}g_{n-1}(1j),\qquad\text{where $j<i$}.
\end{equation}

Otherwise $j>i$, then the occurrences of the pattern 13-2 in the permutation~$\F(\pi)$ starting with the second letter
are the $j-i-1$ subsequences $ij(i+1)$, $ij(i+2)$, $\ldots$, $ij(j-1)$. Thus $f_2(\pi)=j-i-1$.
In this case, the occurrences of the pattern 13-2 in the permutation~$\F(\pi')$ starting with the first letter
are the $j-3$ subsequences $1(j-1)2$, $1(j-1)3$, $\ldots$, $1(j-1)(j-2)$. Thus $f_1(\pi')=j-3$, which implies that
\[
f_1(\pi)+f_2(\pi)-f_1(\pi')
\=(i-2)+(j-i-1)-(j-3)
\=0.
\]
In other words, the permutations~$\F(\pi)$ and~$\F(\pi')$ have the same number of occurrences of the pattern 13-2, i.e.,
\begin{equation}\label{eq:j>i:13-2}
g_n(1ij)\=g_{n-1}\bigl(1(j-1)\bigr),\qquad\text{where $j>i$}.
\end{equation}

By summing (\ref{eq:j<i:13-2}) over $j<i$, and summing (\ref{eq:j>i:13-2}) over $j>i$, we obtain that
\begin{align}
g_n(1i)
&\=\sum_{j<i}g_n(1ij)+\sum_{j>i}g_n(1ij)\notag\\
&\=\sum_{j<i}q^{i-j}g_{n-1}(1j)+\sum_{j>i}g_{n-1}\bigl(1(j-1)\bigr)\notag\\
\label{rec:g1i:13-2}
&\=g_{n-1}+\sum_{j<i}(q^{i-j}-1)g_{n-1}(1j),\qquad\text{where $3\le i\le n$}.
\end{align}
For the sake of cancelling the Sigma notation in the above expression,
we compute the second-order difference transformation of the above formula, which gives that
\begin{align}\label{rec:g1k:13-2}
g_n(1k)\=(1+q) g_n\bigl(1(k-1)\bigr)-q\cdotp g_n\bigl(1(k-2)\bigr)-(1-q)g_{n-1}\bigl(1(k-1)\bigr),
\quad\text{where $5\le k\le n$}.
\end{align}
This is a recurrence relation of the function $g_n(1k)$.
With the aid of Lemma~\ref{lem:g12:13-2}, by taking $i=3$ and $i=4$ in (\ref{rec:g1i:13-2}), we obtain that
\begin{equation}\label{ini:g1k:13-2}
\begin{split}
g_n(13)&\=g_{n-1}-2(1-q)g_{n-2}\quad\text{for $n\ge3$, and}\\
g_n(14)&\=g_{n-1}-(1-q)(3+2q)g_{n-2}+2(1-q)^2g_{n-3}\quad\text{for $n\ge4$}.
\end{split}
\end{equation}

In order to state a ``solution'' of (\ref{rec:g1k:13-2}) with the initial condition (\ref{ini:g1k:13-2}),
we define a sequence $a_{k,j}$ for any $k\ge2$ and for any integer~$j$ recursively, by
$a_{2,j}=\chi(j=1)$, $a_{3,j}=j\cdotp\chi(j\in\{1,2\})$, and
\begin{equation}\label{def:a:13-2}
a_{k,j}\=(1+q)a_{k-1,\,j}-q\cdotp a_{k-2,\,j}+a_{k-1,\,j-1}\quad\text{for $k\ge4$}.
\end{equation}
It follows that $a_{k,j}=0$ if $j\le0$ or $j\ge k$, and that $a_{k,1}=1$.
We claim that
\begin{equation}\label{sol:g1k:13-2}
g_n(1k)=\sum_{j=1}^{k-1}a_{k,j}(q-1)^{j-1}g_{n-j}, \quad 3\le k\le n.
\end{equation}
In fact, by multiplying (\ref{def:a:13-2}) by $(q-1)^jg_{n-j}$
and summing it over all integers $j$ gives (\ref{rec:g1k:13-2}).
It is also straightforward to verify that the initial values of~$g_n(13)$ and~$g_n(14)$ which are given by~(\ref{sol:g1k:13-2})
coincide with (\ref{ini:g1k:13-2}).
Since (\ref{rec:g1k:13-2}) and~(\ref{ini:g1k:13-2}) determine the sequence $g_n(1k)$ uniquely,
we conclude (\ref{sol:g1k:13-2}) immediately.

By Lemma~\ref{lem:g12:13-2} and (\ref{sol:g1k:13-2}),
we can recast the generating function $g_n$ as
\[
g_n
\=g_n(12)+\sum_{k=3}^ng_n(1k)
\=2g_{n-1}+\sum_{k=3}^n\sum_{j=1}^{k-1}a_{k,j}(q-1)^{j-1}g_{n-j}.
\]
Since $a_{k,1}=1$, we can reduce the above formula to
\[
g_n
\=n\cdotp g_{n-1}+\sum_{k=3}^n\sum_{j=2}^{k-1}a_{k,j}(q-1)^{j-1}g_{n-j}
\=n\cdotp g_{n-1}+\sum_{j=2}^{n-1}b_{n,j}(q-1)^{j-1}g_{n-j},
\]
where $b_{n,j}=\sum_{k\le n}a_{k,j}$ for $j\ge2$.
Define
\begin{align*}
A(x,y)&\=\sum_{k\ge2}\sum_{j}a_{k,j}x^{k-2}y^{j-1},\\
B(x,y,z)&\=\sum_{n\ge2}\sum_{k\le n}\sum_{j}a_{k,j}x^{k-2}y^{j-1}z^{n-2}.
\end{align*}
It is routine to calculate that (\ref{def:a:13-2}) is equivalent to the generating function
\[
A(x,y)\=\frac{1-qx+xy}{(1-x)(1-qx)-xy}.
\]
It follows that
\[
B(x,y,z)
\=\sum_{j}\sum_{k\ge2}\sum_{n\ge k}a_{k,j}x^{k-2}y^{j-1}z^{n-2}
\=\frac{A(xz,y)}{1-z}.
\]
Extracting the coefficient of $y^{j-1}z^{n-2}$ from the specification~$B(1,y,z)$,
we obtain the desired expression~(\ref{def:b:13-2}) for the polynomials~$b_{n,j}$. This completes the proof.
\end{proof}

Let $H_n=\sum_{k=1}^n\frac{1}{k}$ be the $n$th harmonic number.
The next corollary is immediate from Theorem~\ref{thm:gn:13-2}.

\begin{corollary}\label{cor:avr:13-2}
For any $n\ge1$,
the average number of
occurrences of the pattern 13-2 in~$\F(\pi)$ over permutations~$\pi$ of length~$n$ is given by
$\frac{n^2+3n+8}{12}-H_n$.
\end{corollary}

\begin{proof}
Differentiating both sides of (\ref{rec:gn:13-2}), and setting $q=1$, we obtain that
\[
g_n'(1)
\=ng_{n-1}'(1)+b_{n,2}g_{n-2}(1)
\=ng_{n-1}'(1)+\frac{(n-2)(n+3)}{6}(n-1)!,
\]
which implies that
\[
\frac{g_n'(1)}{n!}
\=\frac{g_{n-1}'(1)}{(n-1)!}+\frac{(n-2)(n+3)}{6n}.
\]
Note that the average number equals $g_n'(1)/n!$.
By solving the above recurrence, we obtain the desired formula.
\end{proof}

Another immediate corollary is the number of 13-2-avoiding flattened permutations.

\begin{corollary}\label{cor:avoid:13-2}
Let $n\ge1$. The number of permutations $\pi$ of length $n$ such that the associated permutation $\F(\pi)$ avoids the pattern 13-2
is $2^{n-1}$.
\end{corollary}

\begin{proof}
See Appendix~\ref{apd:avoid}.
\end{proof}

\bigskip

\section{Flattened permutations with $r$ occurrences of the pattern 13-2}

In this section,
we focus on flattened permutations with a given number $r$ of occurrences of
the pattern 13-2.
One may get a recurrence for the number of such permutations
by extracting the coefficient of~$q^r$ from (\ref{rec:gn:13-2}).
We will develop another approach which is more convenient to investigate.

For any nonnegative integer~$r$,
let $g_{n,r}(a_1a_2\cdots a_k)$ be the number of permutations~$\pi$
of length~$n$ such that the permutation~$\F(\pi)$ starts
with the letters $a_1a_2\cdots a_k$ and has exactly~$r$ occurrences of the pattern 13-2.
For example, we have $g_{3,1}(12)=0$ and $g_{3,1}(13)=2$.
Let
\[
g_{n,r}\=g_{n,r}(1).
\]
Then the number $g_{n,r}$ counts the permutations~$\pi$ of length~$n$ whose associated flattened permutation~$\F(\pi)$
has exactly~$r$ occurrences of the pattern 13-2.

The following lemma gives the minimum length of a flattened permutation which has $r$ occurrences of the pattern 13-2.

\begin{lem}
Among the flattened permutations of length $n$,
the flattened permutation $1n2(n-1)3(n-2)\cdots$ has the maximum number of occurrences of the pattern 13-2.
Consequently, any flattened permutation having $r$ ($r\ge 1$) occurrences of the pattern 13-2
has length at least $1+2\sqrt{r}$.
\end{lem}
\begin{proof}
Let $S_{a_1a_2\cdots a_k}$ 
be the set of flattened permutations of length $n$, starting with the letters $a_1a_2\cdots a_k$, 
and having the maximum number of occurrences of the pattern 13-2 among the flattened permutations of length $n$.
The desired result can be restated as $S_{1n2(n-1)3(n-2)\cdots}\neq\emptyset$.

Let $\pi=1\pi_2\pi_3\cdots\pi_n$ be a flattened permutation. Suppose that $\pi_2<n$.
Let $\pi'$ be the permutation obtained by exchanging the letters $\pi_2$ and $\pi_2+1$ in $\pi$.
We claim that the permutation $\pi'$ has at least the same number of occurrences of the pattern 13-2 as the permutation $\pi$.
Let $f(\pi,i)$ be the number of subsequences of the pattern 13-2 starting from the $i$th position.
Let $\Delta_i=f(\pi',i)-f(\pi,i)$.
Then we have $\Delta_1=1$ since the subsequence $1(\pi_2+1)\pi_2$ is an occurrence of the pattern 13-2 in the permutation $\pi'$. Next, we have 
\[
\Delta_2\=\begin{cases}
0,&\text{if $\pi_3\le \pi_2+1$};\\
-1,&\text{otherwise}.
\end{cases}
\]
For all $3\le i\le n$, we have $\Delta_i=0$, since that a subsequence $\pi_i\pi_j\pi_k$  ($3\le i<j<k\le n$) is an occurrence of the pattern 13-2 in the permutation $\pi$ if and only if it is an occurrence of the pattern 13-2 in the permutation~$\pi'$.
Consequently, we have $\sum_{1\le i\le n}\Delta_i\in\{0,1\}$, which implies the truth of the claim.
Continuing adjusting the permutation $\pi'$ in this way, 
we obtain that $S_{1n}\neq\emptyset$ eventually.

Let $\sigma=1n\sigma_3\sigma_4\cdots\sigma_n\in S_{1n}$.
Assume that $\sigma_3\ne 2$. 
Let $\sigma'$ be the permutation obtained by exchanging the letters $\sigma_3$ and $\sigma_3-1$.
In the same fashion, we deduce that the permutation $\sigma'$ has at least the same number of occurrences of the pattern 13-2
as the permutation $\sigma$. Continuing the adjustment, we find that $S_{1n2}\ne\emptyset$.

Now, for the same reason as we derive that $S_{1n}\ne\emptyset$, we can show that $S_{1n2(n-1)}\ne\emptyset$.
Next, for the same reason as we derive $S_{1n2}\neq\emptyset$, we infer that $S_{1n2(n-1)3}\ne\emptyset$.
Continuing in this way, we prove that $S_{1n2(n-1)3(n-2)\cdots}\neq\emptyset$.

It is direct to compute that the number of occurrences of the pattern 13-2 in the permutation $1n2(n-1)3(n-2)\cdots$ is 
\[
(n-2)+(n-4)+(n-6)+\cdots\=\begin{cases}
n(n-2)/4,&\text{if $n$ is even};\\
(n-1)^2/4,&\text{if $n$ is odd}.
\end{cases}
\]
Since the permutation $1n2(n-1)3(n-2)\cdots$ has the maximum number of occurrences of the pattern 13-2,
we infer that the above number is larger than or equal to the number $r$.
Solving this inequality out, we find that $n\ge 1+2\sqrt{r}$. This completes the proof.
\end{proof}

Define the generating functions
\begin{equation}\label{def:gf:13-2}
G_{n,r}(v)\=\sum_{i=2}^ng_{n,r}(1i)v^{i-2}\rmand
G_{r}(x,v)\=\sum_{n\ge r+3}G_{n,r}(v)x^n.
\end{equation}
The goal of this section is to show the rationality of the function $G_r(x,v)$.
Observe that a flattened permutation starting with the letters~$1i$ has at least $i-2$ occurrences of the pattern 13-2,
that is, the subsequences $1i2$, $1i3$, $\ldots$, and $1i(i-1)$. It follows that
\[
g_{n,r}(1i)=0\qquad\text{if $i\geq r+3$}.
\]
Therefore, the generating functions defined by (\ref{def:gf:13-2}) can be written equivalently as
\begin{align}
\label{def:Gnr:13-2}
G_{n,r}(v)&\=\sum_{i=2}^{r+2}g_{n,r}(1i)v^{i-2},\\
\label{def:Gr:13-2}
G_{r}(x,v)&
\=\sum_{n\ge r+3}\sum_{i=2}^{r+2}g_{n,r}(1i)v^{i-2}x^{n}
\=\sum_{i=0}^{r}\Biggl(\sum_{n\ge 0}g_{n+r+3,r}\bigl(1(i+2)\bigr)x^{n}\Biggr)x^{r+3}v^{i}.
\end{align}

We remark that the index $n$ is set to be at least $r+3$ in \eqref{def:gf:13-2} for the following reasons.
On the one hand, restricting the lower bound of the index~$n$ to any constant does not affect 
the rational nature of the full generating function. 
On the other hand, the usual setting $n\ge 0$ gives a recurrence more complicated than Lemma~\ref{lem:rec:G:13-2}, which can be seen by adding the sum $\sum_{n=0}^{r+2}G_{n,r}(v)x^n$ to the function $G_r(x,v)$ in \eqref{rec:Gr:13-2}, and which makes the recurrence harder to dealt with.

\medskip

Let $X=\{x_1,x_2,\ldots,x_k\}$ be a set of indeterminates.
For any integral domain~$R$, we denote by~$R[X]$ (resp.,~$R[[X]]$)
the ring of polynomials (resp., formal power series)
in the indeterminates in the set~$X$, with coefficients in the integral domain~$R$.
From (\ref{def:Gr:13-2}), we see that $G_r(x,v)\in x^{r+3}\Z[[x]][v]$, and that
\begin{equation}\label{deg:Gr:13-2}
\deg_vG_r(x,v)\=r.
\end{equation}
As usual, we use the notation~$2\Z[X]$ (resp.,~$2\Z[[X]]$)
stands for the subring of $\Z[X]$ (resp., the subring of~$\Z[[x]]$) such that every coefficient is an even integer.

\begin{lem}\label{lem:2:13-2}
Let $r\ge1$. Then the integer $g_{n,r}(1k)$ is even for all $k\ge2$.
In other words,
we have $G_r(x,v)\in 2x^{r+3}Z[[x]][v]$.
\end{lem}

\begin{proof}
Note that the expression (\ref{def:b:13-2}) of the polynomial $b_{n,j}$ can be recast as
\[
b_{n,j}
=\sum_{k=0}^{n-1-j}\binom{j+k-2}{j-2}
\Biggl(\binom{n-k-2}{j-1}+\binom{n-k-1}{j}\Biggr)q^k.
\]
From this point of view, we see that the coefficients of the polynomials $b_{n,j}$ are integers, namely, $b_{n,j}\in\Z[q]$.
In particular, taking $j=n-1$ in (\ref{def:b:13-2}) gives that $b_{n,\,n-1}=2$.

Now we show that $g_n\in2\Z[q]$ for all $n\ge2$, by induction.
Since $g_2=2$, we can suppose that all the polynomials $g_2,g_3,\ldots,g_{n-1}$ are in the ring $2\Z[q]$.
From the recurrence (\ref{rec:gn:13-2}), we see that the polynomial $g_n$ is a linear combination
of the polynomials $g_1,g_2,\ldots,g_{n-1}$ over the ring $\Z[q]$.
By the induction hypothesis, it suffices to show that the coefficient associated with the polynomial~$g_1$
is in the ring $2\Z[q]$. In fact, this coefficient is $b_{n,\,n-1}q^{n-2}=2q^{n-2}\in 2\Z[q]$.

Consequently, by (\ref{ini:g1k:13-2}), we infer that the polynomials $g_n(13)$ and $g_n(14)$ are in the ring $2\Z[q]$.
Then, by the recurrence (\ref{rec:g1k:13-2}) of the polynomials $g_n(1k)$, we deduce that
\[
g_n(1k)\in2\Z[q],
\]
for all $k\ge5$. From Lemma~\ref{lem:g12:13-2}, we see that the above relation also holds for $k=2$.
In conclusion, the coefficient $g_{n,r}(1k)$ of~$q^r$ in the polynomial $g_n(1k)$ is an even integer for all $k\ge2$.
Hence, by (\ref{def:Gr:13-2}), we deduce that $G_r(x,v)\in 2x^{r+3}\Z[[x,v]]$.
\end{proof}

For convenience, we let $s=1-x$ and $t=1-2x$. For $r\ge1$, define
\begin{equation}\label{def:P:13-2}
P_r(x,v)=\frac{s^{2r-1}t^{r+1}}{2x^{r+3}}\cdotp G_r(x,v).
\end{equation}
By Lemma~\ref{lem:2:13-2} and (\ref{deg:Gr:13-2}), respectively, we have
\begin{equation}\label{deg:P:v:13-2}
P_r(x,v)\in\Z[[x]][v]
\rmand
\deg_v P_r(x,v)=r.
\end{equation}
The main goal of this section is to show that
\[
P_r(x,v)\in\Z[x,v].
\]
We will do this by using a recurrence relation for the series $P_r(x,v)$.
At first, we deduce a recurrence relation for the generating function $G_r(x,v)$.

\begin{lem}\label{lem:rec:G:13-2}
Let $r\ge0$. Then we have
\begin{align}\label{rec:Gr:13-2}
(1-v+vx)G_{r}(x,v)
\=x(1-v)\sum_{j=0}^{r-1}v^{r-j}G_j(x,v)+x(2-v)G_{r}(x,1)+x^3H_r(x,v),
\end{align}
where
\begin{equation}\label{def:H:13-2}
H_r(x,v)
\=x^{r}(2-v)G_{r+2,\,r}(1)-x^{r}vG_{r+2,\,r}(v)-x(1-v)\sum_{n=0}^{r-2}\sum_{j=0}^nv^{r-j}G_{n+3,\,j}(v)x^n.
\end{equation}
\end{lem}

\begin{proof}
Let $r\ge0$. Extracting the coefficient of $q^r$ from (\ref{rec:g1i:13-2}) gives that
\[
g_{n,r}(1i)-g_{n-1,\,r}+\sum_{j\le i-1}g_{n-1,\,r}(1j)-\sum_{j\le i-1}g_{n-1,\,r-i+j}(1j)
\=0,\quad\text{for $3\le i\le n$}.
\]
For each summand in the above formula,
multiplying it by ${v^{i-2}x^{n}}$, and summing it over $3\le i\le r+2$ and over $n\ge r+3$, we obtain that
\begin{align*}
\sum_{n\ge r+3}\sum_{i=3}^{r+2}g_{n,r}(1i)v^{i-2}x^{n}
&\=G_{r}(x,v)-2xG_r(x,1)-2x^{r+3}G_{r+2,\,r}(1),\\
\sum_{n\ge r+3}\sum_{i=3}^{r+2}g_{n-1,\,r}v^{i-2}x^{n}
&\=\frac{vx(1-v^r)}{1-v}\Bigl(G_{r}(x,1)+x^{r+2}G_{r+2,\,r}(1)\Bigr),\\
\sum_{n\ge r+3}\sum_{i=3}^{r+2}\sum_{j=2}^{i-1}g_{n-1,\,r}(1j)v^{i-2}x^{n}
&\=\frac{vx}{1-v}\Bigl(G_r(x,v)+x^{r+2}G_{r+2,\,r}(v)\Bigr)\\
&\qquad-\frac{xv^{r+1}}{1-v}\Bigl(G_r(x,1)+x^{r+2}G_{r+2,\,r}(1)\Bigr),\\
\sum_{n\ge r+3}\sum_{i=3}^{r+2}\sum_{j=2}^{i-1}g_{n-1,\,r-i+j}(1j)v^{i-2}x^{n}
&\=x\sum_{j=0}^{r-1}v^{r-j}G_j(x,v)-x^4\sum_{n=0}^{r-2}\sum_{j=0}^nv^{r-j}G_{n+3,\,j}(v)x^n.
\end{align*}
Adding these equations up,
we derive (\ref{rec:Gr:13-2}).
\end{proof}

A recurrence relation of the series $P_r(x,v)$ can be read off from (\ref{rec:Gr:13-2}).

\begin{lem}\label{lem:rec:P:13-2}
Let $r\ge 1$. Then the series $P_r(x,v)$ satisfies the following recurrence relation:
\begin{equation}\label{rec:P:13-2}
P_r(x,v)
\=2\biggl(\frac{st}{x}\biggr)^{r-1}\cdotp T_r(x,v)
+\sum_{j=1}^{r-1} \frac{\tilde{P}_{r,j}(x,v)}{1-sv}\biggl(\frac{st}{x}\biggr)^{r-j-1}
+\frac{s^{2r-1}t^{r+1}}{2x^r}\cdot\frac{\tilde{H}_r(x,v)}{1-sv},
\end{equation}
where
\begin{align}
\label{def:T:13-2}
T_r(x,v)&\=\frac{x(2-v)+t(1-v)(sv)^r}{1-sv}\;\in\; \Z[x,v],\\
\label{def:tP:13-2}
\tilde{P}_{r,j}(x,v)
&\=st(1-v)(sv)^{r-j}P_j(x,v)+sx(2-v)P_j\Bigl(x,\,\frac{1}{s}\Bigr),\rmand\\
\label{def:tH:13-2}
\tilde{H}_r(x,v)
&\=H_r(x,v)-(2-v)st^{-1}H_r\Bigl(x,\,\frac{1}{s}\Bigr).
\end{align}
\end{lem}

\begin{proof}
We use the kernel method.
Taking $v=s^{-1}$ in (\ref{rec:Gr:13-2}), we obtain that
\[
G_r(x,1)
\=\frac{x}{t}\sum_{j=0}^{r-1}\frac{1}{s^{r-j}}G_j\bigl(x,s^{-1}\bigr)
-\frac{x^2s}{t}H_r\bigl(x,s^{-1}\bigr).
\]
Substituting it back into (\ref{rec:Gr:13-2}) gives that
\begin{equation}\label{rec1:Gr:13-2}
(1-sv)G_{r}(x,v)
\=x(1-v)\sum_{j=0}^{r-1}v^{r-j}G_j(x,v)
+\frac{x^2(2-v)}{t}\sum_{j=0}^{r-1}\frac{1}{s^{r-j}}G_j\bigl(x,s^{-1}\bigr)+x^3\tilde{H}_r(x,v),
\end{equation}
where $\tilde{H}_r(x,v)$ is defined by (\ref{def:tH:13-2}).
By taking $r=0$, we can compute by bootstrapping that
\begin{alignat*}{2}
(\ref{def:Gnr:13-2})
&\quad\Longrightarrow\quad&
G_{2,0}(v)
&\=g_{2,0}(12)\=2,\\[6pt]
(\ref{def:H:13-2})
&\quad\Longrightarrow\quad&
H_0(x,v)
&\=(2-v)G_{2,0}(1)-vG_{2,0}(v)
\=4(1-v),\\
(\ref{def:tH:13-2})
&\quad\Longrightarrow\quad&
\tilde{H}_0(x,v)
&\=H_0(x,v)-\frac{(2-v)s}{t}H_0\Bigl(x,\,\frac{1}{s}\Bigr)
\=\frac{4(1-sv)}{t}.
\end{alignat*}
Hence, by (\ref{rec1:Gr:13-2}), we obtain that
\[
G_0(x,v)
\=\frac{x^3\tilde{H}_0(x,v)}{1-sv}
\=\frac{4x^3}{t}.
\]
Substituting it back into (\ref{rec1:Gr:13-2}), we find that
\begin{multline}\label{rec2:Gr:13-2}
G_{r}(x,v)\=\frac{1}{1-sv}\left[\frac{4x^4}{t}\biggl((1-v)v^r+\frac{x(2-v)}{s^rt}\biggr)+x^3\tilde{H}_r(x,v)\right.\\
\left.+x\sum_{j=1}^{r-1}\biggl((1-v)v^{r-j}G_j(x,v)+\frac{x(2-v)G_j(x,s^{-1})}{s^{r-j}t}\biggr) \right].
\end{multline}
By using (\ref{def:P:13-2}), we can express the series $G_j(x,v)$ in terms of the series $P_j(x,v)$ for $j\ge1$.
Substituting the resulting expression into the above equation, we obtain (\ref{rec:P:13-2}).

We need to show that $T_r(x,v)\in\Z[x,v]$. In fact, by using the summation formula
$\sum_{k=0}^{r-1}(sv)^k=\frac{1-(sv)^{r}}{1-sv}$
for geometric series, we can recast the series $T_r(x,v)$ as
\begin{equation}\label{T:expand:13-2}
T_r(x,v)
\=1-t(1-v)\sum_{k=0}^{r-1}(sv)^k
\=1-(1-2x)(1-v)\sum_{k=0}^{r-1}(1-x)^kv^k
\;\in\; \Z[x,v].
\end{equation}
This completes the proof.
\end{proof}

For the sake of extracting the coefficient of powers of~$v$ from the function $\tilde{H}_r(x,v)$,
we give another expression for $\tilde{H}_r(x,v)$ in the following lemma.

\begin{lem}\label{lem:tH:13-2}
We have
\begin{align}
\label{tH:expand:13-2}
\frac{\tilde{H}_r(x,v)}{1-sv}
&\=\frac{x^r}{t}\cdotp\sum_{i=2}^{r+2}\frac{g_{r+2,\,r}(1i)}{s^{i-1}}\biggl(1+t\sum_{k=0}^{i-2}(sv)^k\biggr)\\
&\qquad-\frac{1}{t}\cdotp\sum_{n=0}^{r-2}\sum_{j=0}^{n}\sum_{k=2}^{j+2}\frac{g_{n+3,\,j}(1k)x^{n+1}}{s^{r-j+k-2}}\cdotp T_{r-j+k-2}(x,v).\notag
\end{align}
\end{lem}

\begin{proof}
See Appendix~\ref{apd:tH}.
\end{proof}

For any formal power series $F\in\Z[[x,v]]$, we use the notation $[v^\ell]F$ to denote the coefficient of~$v^\ell$ in the series~$F$.
It follows that $[v^\ell]F$ is a formal power series in the ring $\Z[[x]]$.

\begin{theorem}\label{thm:coef:13-2}
Let $r\ge 1$. The series $P_r(x,v)$ is a polynomial in the ring $\Z[x,v]$, which can be expressed as
\begin{equation}\label{P:expand:13-2}
P_r(x,v)
\=2c_{r,0}(x)+\sum_{\ell=1}^rc_{r,\ell}(x)(1-x)^{\ell-1}(1-2x)^\ell v^\ell,
\end{equation}
where $c_{r,i}(x)\in\Z[x]$ for all $0\le i\le r$. Moreover, we have
\begin{itemize}
\smallskip\item
$\deg c_{r,0}(x)=3r-1$ for all $r\ge4$, and $c_{r,0}(1/2)=2^{1-r}$ for all $r\ge 1$; and
\smallskip\item
$\deg c_{r,\ell}(x)=3r-2\ell$ for all $r\ge4$ and for all $1\le \ell\le r$.
\end{itemize}
\end{theorem}

\begin{proof}
In view of (\ref{deg:P:v:13-2}), we can define the functions $c_{r,i}(x)$ for $0\le i\le r$ by (\ref{P:expand:13-2}).
Then it suffices to show that the functions $c_{r,i}(x)$ satisfy all the desired properties.
We proceed by induction on $r$.

Setting $r=1$, we can compute by bootstrapping as follows:
\begin{alignat*}{2}
(\ref{def:Gnr:13-2})
&\quad\Longrightarrow\quad&
G_{3,1}(v)
&\=g_{3,1}(12)+g_{31}(13)v
\=2v,\\[5pt]
(\ref{def:H:13-2})
&\quad\Longrightarrow\quad&
H_1(x,v)
&\=x(2-v)G_{31}(1)-xvG_{31}(v)
\=2x(1-v)(2+v),\\
(\ref{def:tH:13-2})
&\quad\Longrightarrow\quad&
\tilde{H}_1(x,v)
&\=H_1(x,v)-(2-v)st^{-1}H_1\bigl(x,s^{-1}\bigr)
\=\frac{2x(1-v+xv)(2+v-2xv)}{st}.
\end{alignat*}
Now, by taking $r=1$ in (\ref{rec2:Gr:13-2}), we infer that
\begin{align*}
G_1(x,v)
\=\frac{1}{1-sv}\left[\frac{4x^4}{t}\biggl((1-v)v+\frac{x(2-v)}{st}\biggr)+x^3\tilde{H}_1(x,v)\right]
\=\frac{2x^4}{st^2}\bigl(2+(3-2x)tv\bigr).
\end{align*}
Therefore, by (\ref{def:P:13-2}), we obtain that
$P_1(x,v)=2+(3-2x)tv$.
Comparing it with (\ref{P:expand:13-2}),
we find that
\begin{equation}\label{c10c11:13-2}
c_{1,0}(x)=1\rmand c_{1,1}(x)=3-2x,
\end{equation}
which satisfy all the desired properties.
Now, we can suppose that the functions $c_{j,i}(x)$ have all the desired properties, where $1\le j\le r-1$ and $0\le i\le j$.
\smallskip

In order to show the desired properties of the functions $c_{r,i}(x)$,
we recast the series $P_r(x,v)$ as $U(x,v)+V(x,v)$, as follows.
This decomposition will be the key to prove all the properties.
Since
\[
P_j(x,v)\=2c_{j,0}(x)+\sum_{\ell=1}^jc_{j,\ell}(x)s^{\ell-1}t^\ell v^\ell,\qquad 1\le j\le r-1,
\]
the function $\tilde{P}_r(x,j)$ defined by \eqref{def:tP:13-2} can be recast as
\begin{align*}
\tilde{P}_{r,j}(x,v)
&\=st(1-v)(sv)^{r-j}P_j(x,v)+sx(2-v)P_j\Bigl(x,\,\frac{1}{s}\Bigr)\\
&\=st(1-v)(sv)^{r-j}\biggl(2c_{j,0}(x)+\sum_{i=1}^jc_{j,i}(x)s^{i-1}t^iv^i\biggr)\\
&\qquad\quad\ \ +sx(2-v)\biggl(2c_{j,0}(x)+\sum_{i=1}^jc_{j,i}(x)s^{i-1}t^i\frac{1}{s^i}\biggr)\\
&\=2s\cdotp c_{j,0}(x)\Bigl[t(1-v)(sv)^{r-j}+x(2-v)\Bigr]
+\sum_{k=1}^jt^k\cdotp c_{j,k}(x)\Bigl[t(1-v)(sv)^{r-j+k}+x(2-v)\Bigr].
\end{align*}
In view of (\ref{def:T:13-2}), the factors in the above brackets can be expressed in terms of the polynomial $T_h(x,v)$:
\[
\frac{\tilde{P}_{r,j}(x,v)}{1-sv}
\=2s\cdotp c_{j,0}(x)\cdot T_{r-j}(x,v)+\sum_{k=1}^jt^k\cdotp c_{j,k}(x)\cdot T_{r-j+k}(x,v).
\]
Substituting the above formula and (\ref{tH:expand:13-2}) into (\ref{rec:P:13-2}),
we obtain the aforementioned decomposition $P_r(x,v)=U(x,v)+V(x,v)$, where
\begin{equation}\label{def:U:13-2}
U(x,v)\=(st)^r\cdot\sum_{j=2}^{r+2}\frac{g_{r+2,\,r}(1j)}{2}s^{r-j}\biggl(1+t\sum_{k=0}^{j-2}(sv)^k\biggr)
\end{equation}
and
\begin{multline}\label{def:V:13-2}
V(x,v)\=\biggl(\frac{st}{x}\biggr)^{\!\!r}\Biggl[
\frac{2x}{st}T_r(x,v)
+\sum_{j=1}^{r-1}\biggl(\frac{x}{st}\biggr)^{j+1}
\biggl(2s\cdotp c_{j,0}(x)T_{r-j}(x,v)\\
+\sum_{k=1}^jt^k\cdotp c_{j,k}(x)T_{r-j+k}(x,v)\biggr)
-\sum_{n=0}^{r-2}\sum_{j=0}^{n}\sum_{k=2}^{j+2}\frac{g_{n+3,\,j}(1k)}{2}
x^{n+1}s^{j-k+1}T_{r-j+k-2}(x,v)\Biggr].
\end{multline}

Below we show the desired properties in steps.
The induction hypotheses will be used when we deal with the function $V(x,v)$.

\medskip
\noindent{\bf Step 1.}
Let $1\le\ell\le r$. We will show that $c_{r,\ell}(x)\in\Z[x]$.
Note that
\[
c_{r,\ell}(x)s^{\ell-1}t^\ell
\=[v^\ell]P_r(x,v)
\=[v^\ell]U(x,v)+[v^\ell]V(x,v).
\]
We will consider the coefficients $[v^\ell]U(x,v)$ and $[v^\ell]V(x,v)$ individually.

By Lemma~\ref{lem:2:13-2}, the integers $g_{n,r}(1k)$ are even.
Thus, from (\ref{def:U:13-2}), we see that
\begin{equation}\label{coef:U:13-2}
[v^\ell]U(x,v)
\=\sum_{i=\ell+2}^{r+2}\frac{g_{r+2,\,r}(1i)}{2}\cdotp s^{\ell+2r-i}t^{r+1}
\;\in\; s^{\ell+r-2}t^{r+1}\Z[x].
\end{equation}

For $[v^{\ell}]V(x,v)$,
we consider the equation (\ref{def:V:13-2}) as
\[
V(x,v)=\sum_{h=1}^r a_h(x)T_h(x,v),
\]
where $a_h(x)\in\Z[[x]]$ by the induction hypothesis that $c_{j,i}(x)\in\Z[x]$ for $j\le r-1$.
We will find out a ring containing the coefficients $[v^\ell]T_h(x,v)$ for all~$1\le h\le r$,
and another ring containing the coefficient $[v^\ell]a_h(x)$ for all~$1\le h\le r$.
By (\ref{T:expand:13-2}), we have
\begin{equation}\label{coef:T:13-2}
[v^i]T_h(x,v)
\=\begin{cases}
2x,&\text{if $i=0$};\\[4pt]
ts^{i-1}x,&\text{if $1\le i\le h-1$},\\[4pt]
ts^{i-1},&\text{if $i=h$},\\[4pt]
0,&\text{else},
\end{cases}
\end{equation}
which implies that
\begin{equation}\label{coef1:T}
[v^\ell]T_h(x,v)\;\in\; s^{\ell-1}t\ \Z[x]\qquad\text{for $1\le\ell\le h$}.
\end{equation}
By the induction hypothesis that $c_{j,i}(x)\in\Z[x]$ for $j\le r-1$, we deduce from (\ref{def:V:13-2}) that
\[
a_h(x)\;\in\; t^{h-1}x^{1-r}\Z[x]\qquad\text{for all $1\le h\le r$}.
\]
To wit, whenever $h\ge \ell$, the coefficient~$a_h(x)$
contributes a series in the ring $t^{h-1}x^{1-r}\Z[x]$ to the coefficient $[v^\ell]V(x,v)$.
Together with (\ref{coef1:T}), we infer that
\begin{equation}\label{coef:V}
[v^\ell]V(x,v)
\=[v^\ell]\sum_{h=1}^ra_h(x)T_h(x,v)
\;\in\; (t^{\ell-1} x^{1-r})\cdot s^{\ell-1}t\;\Z[x]
\= s^{\ell-1}t^{\ell} x^{1-r}\Z[x].
\end{equation}
In view of (\ref{coef:U:13-2}) and (\ref{coef:V}),
since $\ell+r-2\ge \ell-1$ and $r+1>\ell$, we infer that
\begin{equation}\label{coef1:P}
[v^\ell]P_r(x,v)
\;\in\;
s^{\ell-1}t^\ell x^{1-r}\Z[x].
\end{equation}

On the other hand, we have $G_r(x,v)\in 2x^{r+3}Z[[x]][v]$ by Lemma~\ref{lem:2:13-2}. By the definition (\ref{def:P:13-2}) of the series $P_r(x,v)$,
we infer that
\[
P_r(x,v)
\=\frac{s^{2r-1}t^{r+1}}{2x^{r+3}}\cdotp G_r(x,v)
\;\in\;
s^{2r-1}t^{r+1}Z[[x]][v].
\]
Together with (\ref{coef1:P}), we find that the coefficient $[v^\ell]P_r(x,v)$ is in the following intersection of rings:
\[
s^{\ell-1}t^\ell x^{1-r}\Z[x]\ \cap \ s^{2r-1}t^{r+1}Z[[x]].
\]
Since $2r-1>\ell -1\ge0$ and $r+1>\ell\ge 1$, we infer that $[v^\ell]P_r(x,v)\in s^{\ell-1}t^\ell \Z[x]$.
In other words, we have $c_{r,\ell}(x)\in\Z[x]$.

\medskip
\noindent{\bf Step 2.}
We show that $c_{r,0}(x)\in\Z[x]$. We will adopt the trick of ring-intersection again.
Note that $1+t=2s$.
Taking $v=0$ in (\ref{def:U:13-2}), we obtain that
\begin{equation}\label{U:v=0}
U(x,0)
\=(st)^r\cdot\sum_{i=2}^{r+2}\frac{g_{r+2,\,r}(1i)}{2}s^{r-i}(1+t)
\=2(st)^r\cdot\sum_{i=2}^{r+2}\frac{g_{r+2,\,r}(1i)}{2}s^{r-i+1}.
\end{equation}
By Lemma~\ref{lem:2:13-2}, we infer that $U(x,0)\in 2\Z[x]$.

From (\ref{T:expand:13-2}), we see that the constant term of the polynomial $T_r(x,v)$ in $v$ is $1-t=2x$.
Taking $v=0$ in (\ref{def:V:13-2}), we get
\begin{multline}\label{V:v=0}
V(x,0)
\=(2x)\cdot\biggl(\frac{st}{x}\biggr)^{\!\!r}\cdot\Biggl[
\frac{2x}{st}
+\sum_{j=1}^{r-1}\biggl(\frac{x}{st}\biggr)^{j+1}
\biggl(2s\cdotp c_{j,0}(x)
+\sum_{i=1}^jt^i\cdotp c_{j,i}(x)\biggr)\\
-\sum_{n=0}^{r-2}\sum_{j=0}^{n}\sum_{k=2}^{j+2}\frac{g_{n+3,\,j}(1k)}{2}
x^{n+1}s^{j-k+1}\Biggr].
\end{multline}
Recall that $c_{j,i}(x)\in\Z[x]$ from the induction hypothesis.
From the above expression, we infer that $x^{r-1}V(x,0)\in2\Z[x]$. Therefore,
\[
2c_{r,0}(x)
\=P_r(x,0)
\=U(x,0)+V(x,0)
\;\in\;2x^{1-r}\Z[x].
\]
On the other hand, since $P_r(x,v)\in\Z[[x]][v]$, we have $2c_{r,0}(x)=P_r(x,0)\in\Z[[x]]$.
Therefore, the series $2c_{r,0}(x)$
belongs to both of the rings $2x^{1-r}\Z[x]$ and $\Z[[x]]$.
Hence, we infer that $2c_{r,0}(x)\in2\Z[x]$, that is, $c_{r,0}(x)\in\Z[x]$.

\medskip
\noindent{\bf Step 3.}
We show that $c_{r,0}(1/2)=2^{1-r}$ for $r\ge1$. Note that $t=0$ when $x=1/2$.
In fact, taking $x=1/2$ in (\ref{U:v=0}), we find that $U(1/2,0)=0$.
Taking $x=1/2$ in (\ref{V:v=0}), we obtain that $V(1/2,0)=c_{r-1,0}(1/2)$.
Therefore, we have
\[
2c_{r,0}(1/2)
\=U(1/2,0)+V(1/2,0)
\=c_{r-1,0}(1/2).
\]
By iterating the above equation, we obtain that $c_{r,0}(1/2)=2^{1-r}c_{1,0}(1/2)$.
By (\ref{c10c11:13-2}), we have the initiation $c_{1,0}(x)=1$. Hence, we have $c_{r,0}(1/2)=2^{1-r}$ for all $r\ge 1$.

\medskip
\noindent{\bf Step 4.} We show the upper bounds of the degrees $\deg_xc_{r,i}(x)$, namely,
$\deg c_{r,0}(x)\le 3r-1$ and $\deg c_{r,\ell}(x)\le 3r-2\ell$ for $1\le \ell\le r$.
Let $0\le i\le r$.

By (\ref{def:U:13-2}), we infer that
\begin{equation}\label{ub:deg:U:x}
\deg_x [v^i]U(x,v)
\l
\max_{i+2\le k\le r+2}(3r-k+i+1)
\l
3r-(i+2)+i+1
\=
3r-1.
\end{equation}
We claim that
\begin{equation}\label{ub:deg:V:x}
\deg_x [v^i]V(x,v)\l3r-2,\qquad\text{for all $0\le i\le r$.}
\end{equation}
To show it, we compute the upper bound of the four summands in the bracket of (\ref{def:V:13-2}), individually.
From (\ref{coef:T:13-2}), we infer that
\[
\deg_x[v^i]T_h(x,v)\l h.
\]
Together with $\deg_x s=\deg_x t=1$, and by using the induction hypothesis, we can deduce that
\begin{align*}
\deg_x[v^\ell]\frac{2x}{st}T_r(x,v)
&\l r-1,\\[5pt]
\deg_x[v^\ell]\biggl(\frac{x}{st}\biggr)^{j+1}2s\cdotp c_{j,0}(x)T_{r-j}(x,v)
&\l -j+(3j-1)+(r-j)
\l 2(r-1),\\[3pt]
\deg_x[v^\ell]\biggl(\frac{x}{st}\biggr)^{j+1}t^i\cdotp c_{j,i}(x)T_{r-(j-i)}(x,v)
&\l -j-1+i+(3j-2i)+(r-j+i)
\l 2(r-1),\\[7pt]
\deg_x[v^\ell]x^{n+1}s^{j-k+1}T_{r-(j+2-k)}(x,v)
&\l (n+1)+(j-k+1)+(r-j-2+k)\l 2(r-1).
\end{align*}
Therefore, by (\ref{def:V:13-2}), we obtain that
\[
\deg_x[v^\ell]V(x,v)
\l r+2(r-1)
\=3r-2.
\]
This proves the claim.

Now, we can infer that
\[
\deg_x[v^\ell]P_r(x,v)
\l
\max(\deg_x[v^\ell]U(x,v),\,\deg_x[v^\ell]V(x,v))
\l 3r-1.
\]
Consequently, by (\ref{P:expand:13-2}), we obtain immediately
that $\deg_xc_{r,0}(x)\le 3r-1$, and that
\[
\deg_xc_{r,\ell}(x)
\l (3r-1)-(2\ell-1)
\=3r-2\ell.
\]

\medskip
\noindent{\bf Step 5.} We show the upper bounds can be attained as desired.
From Step 4, we see that the upper bound for $\deg_x c_{r,\ell}(x)$ can be attained if and only if
\[
\deg_x [v^\ell]U(x,v)
\=3r-1.
\]
From (\ref{ub:deg:U:x}), we see that the above equation holds if and only if $i=\ell+2$ and $g_{r+2,\,r}(1i)\ne0$, that is,
\[
g_{r+2,\,r}\bigl(1(\ell+2)\bigr)
\;\ne\;
0.
\]
In fact, the above inequality is true for all $r\ge4$; a simple combinatorial proof is provided in Appendix~\ref{apd:0}.
This completes the proof.
\end{proof}

\begin{corollary}\label{cor:rt:13-2}
The generating function
\[
G_r(x,v)
\=\sum_{n\ge r+3}\sum_{i=2}^ng_{n,r}(1i)v^{i-2}x^n
\=\frac{2x^{r+3}P_r(x,v)}{(1-x)^{2r-1}(1-2x)^{r+1}}
\]
is rational, where $P_r(x,v)\in\Z[x,v]$.
\end{corollary}

\begin{proof}
Immediate from Theorem~\ref{thm:coef:13-2} and the definition (\ref{def:P:13-2}) of the polynomial $P_r(x,v)$.
\end{proof}

We can compute $c_{r,i}(x)$ by using Lemma~\ref{lem:rec:G:13-2}.
Here we list them for $1\le r\le 5$.
\begin{align*}
c_{1,0}(x)=&1,\qquad c_{1,1}(x)=3-2x,\allowbreak \\[5pt]
c_{2,0}(x)=&3-6x+2x^2,\qquad c_{2,1}(x)=5-10x+4x^2,\qquad c_{2,2}(x)=10-15x+6x^2,\allowbreak \\[5pt]
c_{3,0}(x)=&12-52x+78x^2-48x^3+12x^4,\allowbreak \\
c_{3,1}(x)=&18-76x+112x^2-68x^3+16x^4,\allowbreak \\
c_{3,2}(x)=&27-95x+128x^2-94x^3+40x^4-8x^5,\\
c_{3,3}(x)=&35-84x+70x^2-20x^3,\\[5pt]
c_{4,0}(x)=&68-544x+2011x^2-4854x^3+8938x^4-12986x^5+14422x^6-11780x^7\\
&+6800x^8-2624x^9+608x^{10}-64x^{11},\\
c_{4,1}(x)=&88-636x+1995x^2-3754x^3+5074x^4-5430x^5+4562x^6-2816x^7\\
&+1168x^8-288x^9+32x^{10},\\
c_{4,2}(x)=&122-770x+2123x^2-3506x^3+3940x^4-3072x^5+1584x^6-480x^7+64x^8,\\
c_{4,3}(x)=&140-691x+1434x^2-1665x^3+1151x^4-448x^5+76x^6,\\
c_{4,4}(x)=&126-420x+540x^2-315x^3+70x^4,\\[5pt]
c_{5,0}(x)=&473-5812x+34630x^2-134895x^3+384546x^4-838332x^5+1416868x^6-1859729x^7\\
&+1888165x^8-1468200x^9+858300x^{10}-365200x^{11}+106800x^{12}-19200x^{13}\\
&+1600x^{14},\\
c_{5,1}(x)=&559-6222x+32664x^2-109535x^3+265560x^4-490864x^5+701932x^6-773549x^7\\
&+648879x^8-406058x^9+183512x^{10}-56608x^{11}+10672x^{12}-928x^{13},\\
c_{5,2}(x)=&690-6656x+29713x^2-82642x^3+160896x^4-229588x^5+242120x^6-186355x^7\\
&+101592x^8-37128x^9+8160x^{10}-816x^{11},\\
c_{5,3}(x)=&771-6237x+22806x^2-50268x^3+74031x^4-75151x^5+52111x^6-23586x^7\\
&+6276x^8-744x^9,\\
c_{5,4}(x)=&693-4316x+11679x^2-18049x^3+17237x^4-10148x^5+3396x^6-496x^7,\\
c_{5,5}(x)=&462-1980x+3465x^2-3080x^3+1386x^4-252x^5.
\end{align*}
\smallskip

We did not succeed in finding explicit expressions for the numbers $c_{r,i}$.

\newpage
\appendix

\section{Proof of Corollary \ref{cor:avoid:13-2}}\label{apd:avoid}

Let $f_n$ be the number of permutations $\pi$ of length $n$ such that the associated permutation $\F(\pi)$ avoids the pattern 13-2.
Taking $q=0$ in (\ref{rec:gn:13-2}), we obtain the recurrence relation
\[
f_n=\sum_{j=1}^{n-1}\frac{n-1+j}{j}\binom{n-2}{j-1}(-1)^{j-1}f_{n-j}\qquad\text{for $n\ge2$},
\]
with $f_1=1$. To show that $f_n=2^{n-1}$, it suffices to check it is a solution of the above recurrence for~$f_n$, namely,
\[
2^n=\sum_{j=1}^{n-1}\frac{n-1+j}{j}\binom{n-2}{j-1}(-1)^{j-1}2^{n-j}.
\]
In fact, we have
\begin{align*}
&\sum_{j=1}^{n-1}\frac{n-1}{j}\binom{n-2}{j-1}(-1)^{j-1}2^{n-j}-2^n\\
&\=\sum_{j=1}^{n-1}\binom{n-1}{j}(-1)^{j-1}2^{n-j}-2^n
\=\sum_{j=0}^{n-1}\binom{n-1}{j}(-1)^{j-1}2^{n-j}\\
&\=-2^n\sum_{j=0}^{n-1}\binom{n-1}{j}\biggl(-\frac{1}{2}\biggr)^{j}
\=-2^n\biggl(1-\frac{1}{2}\biggr)^{n-1}\=-2,
\end{align*}
and
\begin{align*}
\sum_{j=1}^{n-1}\frac{j}{j}\binom{n-2}{j-1}(-1)^{j-1}2^{n-j}
\=2^{n-1}\sum_{j=1}^{n-1}\binom{n-2}{j-1}\biggl(-\frac{1}{2}\biggr)^{j-1}
\=2^{n-1}\biggl(1-\frac{1}{2}\biggr)^{n-2}\=2.
\end{align*}
Adding the above two equations up, we verified the desired identity.
This completes the proof for Corollary \ref{cor:avoid:13-2}.
\smallskip

A rather simple combinatorial proof is as follows.
For any flattened permutation $\F(\pi)$ of length~$n$ avoiding the pattern 13-2,
the letter~$n$ must sit in its last position.
Thus, the letter $n$ in the permutation $\pi$ is also placed in the last position,
which either forms a single cycle, or occupies the last position of the last cycle of $\pi$.
Therefore, the numbers $f_n$ satisfy that $f_n=2f_{n-1}$ for $n\ge 2$.
Since $f_1=1$, we infer that $f_n=2^{n-1}$ immediately.

\medskip

\section{Proof of Lemma \ref{lem:tH:13-2}}\label{apd:tH}

Lemma \ref{lem:tH:13-2} is an expression for the function $\tilde{H}_r(x,v)$.
In view of (\ref{def:H:13-2}), the function $H_r(x,v)$ is a sum of three summands, namely,
\begin{equation}\label{H=A+B+C}
H_r(x,v)
\=A(v)+B(v)+C(v),
\end{equation}
where
\begin{align*}
&A(v)\=x^{r}(2-v)G_{r+2,\,r}(1),\\[10pt]
&B(v)\=-x^{r}vG_{r+2,\,r}(v),\\
&C(v)\=-x(1-v)\sum_{n=0}^{r-2}\sum_{j=0}^nv^{r-j}G_{n+3,\,j}(v)x^n.
\end{align*}
Substituting (\ref{H=A+B+C}) into (\ref{def:tH:13-2}), we can deduce that
\begin{equation}\label{tH=dA+dB+dC}
\tilde{H}_r(x,v)
\=\bigl[A(v)-h\cdotp A(s^{-1})\bigr]
+\bigl[B(v)-h\cdotp B(s^{-1})\bigr]
+\bigl[C(v)-h\cdotp C(s^{-1})\bigr],
\end{equation}
where $h=(2-v)st^{-1}$. Recall that $s=1-x$ and $t=1-2x$.
It follows that $t-(2s-1)=0$. Thus,
\begin{align}
A(v)-h\cdotp A(s^{-1})
\notag
&\=x^{r}(2-v)G_{r+2,\,r}(1)-(2-v)st^{-1}\cdot x^{r}(2-s^{-1})G_{r+2,\,r}(1)\\
\notag
&\=x^{r}(2-v)G_{r+2,\,r}(1)t^{-1}\bigl[t-(2s-1)\bigr]\\
\label{pf:dA}
&\=0.
\end{align}
By (\ref{def:Gnr:13-2}), we can compute
\begin{align*}
B(v)-h\cdotp B(s^{-1})
&\=-x^{r}vG_{r+2,\,r}(v)+h\cdotp x^{r}s^{-1}G_{r+2,\,r}\bigl(s^{-1}\bigr)\\[8pt]
&\=-x^{r}v\sum_{i=2}^{r+2}g_{r+2,\,r}(1i)v^{i-2}+(2-v)st^{-1}x^{r}s^{-1}\sum_{i=2}^{r+2}g_{r+2,\,r}(1i)s^{2-i}\\
&\=\frac{x^r}{t}\cdotp\sum_{i=2}^{r+2}\frac{g_{r+2,\,r}(1i)}{s^{i-1}}\bigl[-t(sv)^{i-1}+(2-v)s\bigr].
\end{align*}
By the summation formula $\sum_{k=0}^{i-2}(sv)^k=\frac{1-(sv)^{i-1}}{1-sv}$
of geometric progressions, we infer that
\begin{align}
\frac{B(v)-h\cdotp B(s^{-1})}{1-sv}
\notag
&\=\frac{x^r}{t}\cdotp\sum_{i=2}^{r+2}\frac{g_{r+2,\,r}(1i)}{s^{i-1}}\cdot\frac{-t(sv)^{i-1}+(2-v)s}{1-sv}\\
\notag
&\=\frac{x^r}{t}\cdotp\sum_{i=2}^{r+2}\frac{g_{r+2,\,r}(1i)}{s^{i-1}}
\biggl[{-t\biggl(\frac{1}{1-sv}-\sum_{k=0}^{i-2}(sv)^k\biggr)}+\frac{(2-v)s}{1-sv}\biggr]\\
\label{pf:dB}
&\=\frac{x^r}{t}\cdotp\sum_{i=2}^{r+2}\frac{g_{r+2,\,r}(1i)}{s^{i-1}}
\biggl[t\sum_{k=0}^{i-2}(sv)^k+1\biggr].
\end{align}
Also, we have that
\begin{align*}
C(v)-h\cdotp C(s^{-1})
&\=-x(1-v)\sum_{n=0}^{r-2}\sum_{j=0}^nv^{r-j}G_{n+3,\,j}(v)x^n
+h\cdotp x(1-s^{-1})\sum_{n=0}^{r-2}\sum_{j=0}^ns^{j-r}G_{n+3,\,j}(s^{-1})x^n\\
&\=\sum_{n=0}^{r-2}x^{n+1}\sum_{j=0}^n
\Bigl[-(1-v)v^{r-j}G_{n+3,\,j}(v)
+(2-v)st^{-1}\cdotp (1-s^{-1})s^{j-r}G_{n+3,\,j}(s^{-1})\Bigr]\\
&\=\sum_{n=0}^{r-2}\frac{x^{n+1}}{t}\sum_{j=0}^n
\Bigl[-t(1-v)v^{r-j}G_{n+3,\,j}(v)+(2-v)(s-1)s^{j-r}G_{n+3,\,j}(s^{-1})\Bigr].
\end{align*}
By (\ref{def:Gnr:13-2}), we can write the summand of the inner sum as
\begin{align*}
&-t(1-v)v^{r-j}\sum_{k=2}^{j+2}g_{n+3,\,j}(1k)v^{k-2}+(2-v)(s-1)s^{j-r}\sum_{k=2}^{j+2}g_{n+3,\,j}(1k)s^{2-k}\\
&\qquad=\sum_{k=2}^{j+2}g_{n+3,\,j}(1k)\Bigl[-t(1-v)v^{r-j}v^{k-2}+(2-v)(s-1)s^{j-r}s^{2-k}\Bigr]\\
&\qquad=\sum_{k=2}^{j+2}\frac{g_{n+3,\,j}(1k)}{s^{r-j+k-2}}\Bigl[-t(1-v)(sv)^{r-j+k-2}+(2-v)(s-1)\Bigr].
\end{align*}
Since $j\le n\le r-2$, we infer that $r-j+k-2\ge k\ge 2$.
Thus, by (\ref{def:T:13-2}), the sum in the above bracket can be recast as
$(sv-1)T_{r-j+k-2}(x,v)$.
Hence, we obtain that
\[
\frac{C(v)-h\cdotp C(s^{-1})}{1-sv}
=-\sum_{n=0}^{r-2}\frac{x^{n+1}}{t}\sum_{j=0}^n
\sum_{k=2}^{j+2}\frac{g_{n+3,\,j}(1k)}{s^{r-j+k-2}}\cdotp T_{r-j+k-2}(x,v).
\]
Substituting (\ref{pf:dA}), (\ref{pf:dB}), and the above formula into (\ref{tH=dA+dB+dC}),
we obtain the desired expression. This completes the proof of Lemma~\ref{lem:tH:13-2}.
\medskip

\section{The bounds for $\deg_x c_{r,i}(x)$ are sharp}\label{apd:0}

We will show that $g_{r+2,\,r}(1(i+2))\ge1$ for all $r\ge4$ and $0\le i\le r$.

From the combinatorial definition of the function $g_{r+2,\,r}(1(i+2))$, it suffices to find a permutation~$\pi$ of length $r+2$
starting with the two letters $1(i+2)$ such that the permutation~$\pi$ has exactly~$r$ occurrences of the pattern 13-2.
When $i=r$, the permutation $\pi=1(r+2)(r+1)\cdots2$ is qualified. Let $0\le i\le r-1$. Define
$\pi=1(i+2)\pi_3\pi_4\cdots\pi_{r+2}$ to be the permutation of length $r+2$ such that
\begin{itemize}
\smallskip
\item
the letters $\pi_r,\pi_{r+1},\pi_{r+2}$ are the three smallest integers in the set $\{1,2,\ldots,r+2\}\setminus\{1,i+2\}$;
\smallskip
\item
the subsequence $\pi_r\pi_{r+1}\pi_{r+2}$ is an occurrence of the pattern 13-2;
\smallskip
\item
the subsequence $\pi_3\pi_4\cdots\pi_{r-1}$ is in the decreasing order, i.e., $\pi_3>\pi_4>\cdots>\pi_{r-1}$.
\end{itemize}
Since $i\le r-1$, we infer that $\pi_3=r+2$.
It is clear that the above rules determine a unique permutation~$\pi$.
It is easy to see that only the three letters $1$, $i+2$, and $\pi_r$ initiate occurrences of the patter 13-2, where
\begin{itemize}
\smallskip
\item
the letter $1$ initiates the $i$ occurrences $1(i+2)j$ for $2\le j\le i+1$;
\smallskip
\item
the letter $i+2$ initiates the $r-i-1$ occurrences $(i+2)(r+2)j$ for $i+3\le j\le r+1$;
\item
the letter $\pi_r$ initiates the occurrence $\pi_r\pi_{r+1}\pi_{r+2}$.
\end{itemize}
In total, the permutation $\pi$ has exactly $i+(r-i-1)+1=r$ occurrences of the pattern 13-2.
This completes the proof.

\bigskip


\bigskip





\begin{thebibliography}{99}



\bibitem{C}
D. Callan,
Pattern avoidance in ``flattened'' partitions,
{\em Discrete Math.}~{\bf 309(12)} (2009), 4187--4191.

\bibitem{HM}
S. Heubach, T. Mansour,
{\em Combinatorics of Compositions and Words},
CRC Press, Boca Raton, 2009.

\bibitem{Ki}
S. Kitaev,
{\em Patterns in Permutations and Words},
Springer-Verlag, Heidelberg, 2011.

\bibitem{M}
T. Mansour,
{\em Combinatorics of Set Partitions},
CRC Press, Boca Raton, 2012.

\bibitem{MS}
T. Mansour and M. Shattuck,
Pattern avoidance in flattened permutations,
{\em Pure Math. Appl.}~{\bf 22(1)} (2011), 75--86.

\bibitem{MSW}
T. Mansour, M. Shattuck and D. G. L. Wang,
Counting subwords in flattened permutations,
{\em J. Combin.}~{\bf 4(3)} (2013), 327--356.

\bibitem{MSW14}
T. Mansour, M. Shattuck and D.G.L. Wang,
Recurrence relations for patterns of type $(2,1)$ in flattened permutations,
{\em J. Difference Equ. Appl.}~{\bf 20(1)} (2014), 58--83.

\bibitem{MV02}
T. Mansour, A. Vainshtein,
Counting occurrences of $132$ in a permutation,
{\em Adv. in Appl. Math.}~{\bf 28(2)} (2002), 185--195.

\bibitem{Nak13D}
B. Nakamura,
Computational methods in permutation patterns,
Ph. D. dissertation at Rutgers University, 2013.

\bibitem{NZ96}
J. Noonan, D. Zeilberger,
The enumeration of permutations with a prescribed number of ``forbidden'' patterns,
{\em Adv. in Appl. Math.}~{\bf 17(4)} (1996) 381--407.
\end{thebibliography}
\end{document}